\documentclass[reqno]{amsart}

\usepackage{latexsym,upref,enumerate}
\usepackage{enumitem} 
\usepackage{nicefrac}
\usepackage{mathtools}
\usepackage{cases}
\usepackage[final]{microtype} %prevent overfull hbox
\usepackage{color}
\usepackage[pdftex,dvipsnames]{xcolor}  % Coloured text etc.
\usepackage[colorinlistoftodos,prependcaption,textsize=tiny]{todonotes}
\usepackage{xargs}

\newcommand{\be}{\begin{equation}}
\newcommand{\ee}{\end{equation}}
\newcommand{\bew}{\begin{equation*}} % equation without numbering
\newcommand{\eew}{\end{equation*}}

\newcommand{\N}{\mathbb N}
\newcommand{\R}{\mathbb R}

\newtheorem{theorem}{Theorem}[section]
\newtheorem{lemma}[theorem]{Lemma}

\theoremstyle{definition}

\numberwithin{equation}{section}

\begin{document}

 \title{On the Cauchy Problem for the Fractional Camassa-Holm Equation}

\author{N. Duruk Mutluba\c{s}}
    \email{nilaydm@sabanciuniv.edu}

\address{Faculty of Engineering and Natural Sciences, Sabanc{\i} University, Orta Mahalle, 34956 Tuzla, Istanbul, Turkey}

\begin{abstract}
In this paper, we consider the Cauchy problem for the fractional Camassa-Holm equation which models  the propagation of small-but-finite amplitude long unidirectional  waves  in a nonlocally and nonlinearly elastic medium. Using Kato's semigroup approach for quasilinear evolution equations, we prove that the Cauchy problem is locally well-posed for data in $H^{s}({\Bbb R})$, $s>{\frac{5}{2}}$. 
\end{abstract}

\maketitle

\setcounter{equation}{0}
\section{Introduction}
\noindent

In the present paper, we prove that the Cauchy problem defined for the fractional Camassa-Holm (fCH) equation: 
\begin{align}
   u_t&+u_x+ uu_x+\frac{3}{4} (-\partial_x^2)^{\nu}u_{x}+\frac{5}{4} (-\partial_x^2)^{\nu}u_{t} \nonumber \\ 
   & +\frac{1}{4} [2(-\partial_x^2)^{\nu}(uu_x)+u(-\partial_x^2)^{\nu}u_x]=0 \label{fCH} \\
 &   u(x,0)=u_0(x) \label{IC}
\end{align}
is locally well-posed for sufficiently smooth initial data. Here, $\nu\geq 1$ is a constant which may not be an integer. 

The fCH equation was introduced in \cite{erbay1} as an asymptotic model equation describing the propagation of small-but-finite amplitude, long unidirectional waves in a one-dimensional infinite, homogeneous medium  made of  nonlocally and nonlinearly elastic material. In the absence of body forces the equation of motion is
\bew
\rho_{0}u_{tt}=(S(u_{X}))_{X}
\eew
where $\rho_{0}$ is the mass density of the medium, $S=S(u_{X})$ is the stress and $u(X,t)$ is the displacement at time $t$ of a reference point $X$ \cite{duruk1}. Nonlocality requires the stress $S$ to be expressed as a general nonlinear nonlocal function of the strain $u_X$ since the stress at a reference point is a nonlinear function of the strain at all points in the body:
 \bew
        S(X,t)=\int_{-\infty}^{\infty} \alpha(|X-Y|)\sigma(Y,t)\mbox{d}Y, ~~~~~
        \alpha :   \mbox{ kernel}
 \eew
After non-dimensionalization and appropriate change in notation, the equation of motion becomes
\be
u_{tt}=[\beta \ast (u+g(u))]_{xx}.
\ee
Here, the constitutive behavior of the medium is described by the convolution integral operator with the general kernel function $\beta(x)$ which is nonnegative, even function, monotonically decreasing for $x>0$, see \cite{duruk2} for the most common used kernel functions in the literature. While deriving fCH in \cite{erbay1}, a fractional-type kernel function was used, that is when the Fourier transform of the kernel function has fractional powers, i.e.  $ \hat{\beta}(\xi)=(1+\xi^{2\nu})^{-1}~$, where $\nu\geq 1$ and may not be an integer. This type of kernel function enables to derive more general evolution equations, in particular fCH equation. It is worth noting that when $\nu=1$, (\ref{fCH}) reduces to the classical Camassa-Holm (CH) equation
 \begin{equation}
      u_{t}+\kappa_{1} (u_{x}-u_{xx})+3 uu_{x}-u_{xxt}=\kappa_{2}(2 u_{x}u_{xx}+uu_{xxx}), \label{c-h}
 \end{equation}
which models the propagation of unidirectional small-amplitude shallow-water waves\cite{camassa, CE1, constantin1, ionescu, johnson2, lannes}. It is an infinite-dimensional completely integrable Hamiltonian system \cite{ constantin4, constantin3} and describes the evolution of the horizontal fluid velocity at a certain depth when nonlinear effects dominate dispersive effects\cite{johnson1}. Therefore, even for smooth initial data, there exist solutions which develop singularities in finite time in the form of wave breaking, i.e. the solution remains bounded but its slope becomes unbounded in finite time  \cite{CE2}. In addition to the studies about water waves, we refer the reader to \cite{erbay1} and the references  therein for the derivation of the CH equation as an appropriate model equation for nonlinear dispersive elastic waves. 
 
In \cite{erbay1}, as a by-product of the asymptotic derivation, the fractional Korteweg-de Vries (fKdV) and fractional  Benjamin-Bona-Mahony (fBBM) equations were also obtained. The fKdV and the fBBM equations are 
 \begin{equation}
     u_{t}+ u_{x}+ uu_{x}-{1\over 2}(-\partial_x^2)^{\nu} u_{x}=0
         \label{fkdv-or}
 \end{equation}
and
 \begin{equation}
       u_{t}+ u_{x}+ uu_{x}+{3\over 4}(-\partial_x^2)^{\nu}u_{x}+{5\over 4}(-\partial_x^2)^{\nu}u_{t}=0,
      \label{fbbm-or}
 \end{equation}
respectively.

Cauchy problems corresponding to the fKdV and fBBM equations have also been studied and qualitative properties of the solutions have been analyzed \cite{HJ, johnson, KS1,  KS2, LPS1, LPS2, LS, Pava}.

To our knowledge, no result concerning the Cauchy problem for the fCH equation is available in the literature. In order to analyze the qualitative properties of the solutions, existence of solutions is a prerequisite. Therefore, the question of the local well-posedness arises naturally. The aim of the present study is to prove, using Kato's theory, that the Cauchy problem (\ref{fCH})-(\ref{IC}) is locally well posed: 

 \begin{theorem}
 \label{Lwp}
   Let $u_0\in H^{s}$, $s>\frac{5}{2}$ be given. Then there exists a maximal time of existence $T>0$, depending on $u_0$, such that there is a unique solution $u$ to
   (\ref{fCH})-(\ref{IC}) satisfying
   \begin{equation*}
    u\in C([0,T), H^{s})\cap C^1([0,T),H^{s-1}).
   \end{equation*}
Moreover, the map $u_0\in H^{s}	\rightarrow u $ is continuous from $H^{s}$ to \\
\mbox{$C([0,T),H^{s})\cap C^1([0,T),H^{s-1})$}.
  \end{theorem}

The paper is organized as follows.  Section 2 presents the approach proposed by Kato using semigroup theory for quasi-linear equations. In Section 3, the local well-posedness of the Cauchy problem (\ref{fCH})-(\ref{IC}) is proved in the light of Kato's theory and suitable reformulations.

\setcounter{equation}{0}
\section{Semigroup Approach}\label{Katotheory}

\noindent
Consider the abstract quasi-linear evolution equation in the Hilbert space $X$:
  \begin{equation}
  \label{qlee}
     u_t +A(u)u = f(u), ~~~~t\geq 0,~~~~~u(0)=u_0.
  \end{equation}
Let $Y$ be a second Hilbert space such that $Y$ is continuously and densely injected into $X$ and let $S:Y\rightarrow X$ be a topological isomorphism. Assume that

\begin{itemize}
\item[(A1)]  For any given $r>0$ it holds that for all $u \in  \mathrm B_r(0) \subseteq Y$ (the ball around the origin in $Y$ with radius $r$), the linear operator $A(u)\colon X \to X$ generates a strongly continuous semigroup $T_u(t)$ in $X$ which satisfies
\begin{equation*}
\| T_u(t) \|_{\mathcal L(X)} \leq \mathrm e^{\omega_r t} \quad \text{for all} \quad t\in [0,\infty)
\end{equation*}
for a uniform constant $\omega_r > 0$.
\item[(A2)] $A$ maps $Y$ into $\mathcal L(Y,X)$, more precisely the domain $D(A(u))$ contains $Y$ and the restriction $A(u)|_Y$ belongs to $\mathcal L(Y,X)$ for any $u\in Y$. Furthermore $A$ is Lipschitz continuous in the sense that for all $r>0$ there exists a constant $C_1$ which only depends on $r$ such that
\begin{equation*}
\| A(u) - A(v) \|_{\mathcal L(Y,X)} \leq C_1 \, \|u-v\|_X 
\end{equation*}
for all $u,~v$ inside $\mathrm B_r(0) \subseteq Y$.
\item[(A3)] For any $u\in Y$ there exists a bounded linear operator $B(u) \in \mathcal L(X)$ satisfying $B(u) = S A(u) S^{-1} - A(u)$ and $B \colon Y \to \mathcal L(X)$ is uniformly bounded on bounded sets in $Y$. Furthermore for all $r>0$ there exists a constant $C_2$ which depends only on $r$ such that  
\begin{equation*}
\| B(u) - B(v)\|_{\mathcal L(X)} \leq C_2 \, \|u-v\|_Y
\end{equation*}
for all $u,~v \in \mathrm B_r(0)\subseteq Y$.
\item[(A4)] The map $f\colon Y \to Y$ is locally $X$-Lipschitz continuous in the sense that for every $r>0$ there exists a constant $C_3>0$, depending only on $r$, such that
\begin{equation*}
\| f(u) - f(v)\|_{X} \leq C_3 \, \|u-v\|_X \quad \text{for all} \; u,~v \in \mathrm B_r(0) \subseteq Y
\end{equation*}
and locally $Y$-Lipschitz continuous in the sense that for every $r>0$ there exists a constant $C_4>0$, depending only on $r$, such that
\begin{equation*}
\| f(u) - f(v)\|_{Y} \leq C_4 \, \|u-v\|_Y \quad \text{for all} \; u,~v \in \mathrm B_r(0) \subseteq Y.
\end{equation*}
\end{itemize}
%Here $C_1$, $C_2$, $C_3$ and $C_4$ are constants depending on $\max\{||u||_Y, ||v||_Y \}$.

\begin{theorem} \cite{KatoI}
\label{kato}
Assume that (A1)-(A4) hold. Then for given $u_0\in Y$, there is a maximal time of existence $T>0$, depending on $u_0$, and a unique solution $u$ to (\ref{qlee}) in $X$ such that
\begin{displaymath}
u=u(u_0,.)\in C([0,T),Y)\cap C^1([0,T),X).
\end{displaymath}
Moreover, the solution depends continuously on the initial data,\\
i.e. the map $u_0\rightarrow u(u_0,.) $ is continuous from $Y$ to \mbox{$C([0,T),Y)\cap C^1([0,T),X)$}.
\end{theorem}

\setcounter{equation}{0}
\section{Proof of Theorem \ref{Lwp}}

\noindent
Recall that fCH equation (\ref{fCH}) is given as:
\begin{displaymath}
 u_t+u_x+ uu_x+\frac{3}{4} (-\partial_x^2)^{\nu}u_{x}+\frac{5}{4} (-\partial_x^2)^{\nu}u_{t} 
    +\frac{1}{4} [2(-\partial_x^2)^{\nu}(uu_x)+u(-\partial_x^2)^{\nu}u_x]=0.
\end{displaymath}
%We first state that (\ref{fCH}) can be rewritten in quasi-linear equation form:
%\bew
%u_t + \tilde A (u) u = \tilde f(u),
%\eew
%with
%\bew
%\tilde A(u) = \frac{2}{5} u\partial_x+\frac{3}{5} \partial_x+\frac{1}{4}\tilde{\Lambda}^{-2\nu} u(-\partial_x^2)^{\nu}\partial_x 
%\eew
%and 
%\bew
%\tilde f(u) =-\tilde{\Lambda}^{-2\nu}\partial_x[\frac{2}{5}u+\frac{3}{10}u^2].
%\eew
%Here $\tilde{\Lambda}=(1+\frac{5}{4}(-\partial_x^2)^{\nu})^{1/2\nu}$ where $\nu\geq 1$ may not be an integer. 
Since particular choice of the constant coefficients 
%-except the 2:1 ratio of the coefficients of the last two terms- 
do not have an impact on local well-posedness, they can be neglected and (\ref{fCH}) can be rewritten in quasi-linear equation form:   
\bew
u_t +  A (u) u =  f(u),
\eew
where
\be \label{ql}
 A(u) = \big(1 +  u +  \Lambda^{-2\nu} [u,(-\partial_x^2)^{\nu}] \big) \partial_x = a(u)\partial_x
\ee
with 
\bew
\Lambda = (1+(-\partial_x^2)^{\nu})^{1/2\nu}, \quad \text{where $\nu\geq 1$ may not be an integer} \quad
\eew
and 
\be \label{nl}
 f(u) =  \Lambda^{-2\nu} \partial_x \big( u^2 \big) .
\ee
Here, $[,]$ represents the usual commutator of the linear operators. 

Due to this reformulation, we choose Hilbert spaces $X\coloneqq (H^{s-1}, \|\cdot \|_{s-1})$ and $Y\coloneqq (H^{s},\|\cdot \|_s)$ with $s>2\nu+\frac{1}{2}\geq \frac{5}{2}$, define the isomorphism $S=\Lambda:Y\rightarrow X$ and prove the lemmas ensuring the validity of the assumptions (A1)-(A4).

The commutator estimate which will be used while proving the assumptions is given by the following lemma:

\begin{lemma}[\cite{Taylor}]
\label{L-comm} 
Let $m> 0$, $s\geq 0$ and $3/2<s+m\leq \sigma$. Then for all $f\in H^{\sigma}$ and $g\in H^{s+m-1}$ one has
\begin{displaymath}
||[\Lambda^{m},f]g||_{s} \leq C||f||_{\sigma} \, ||g||_{s+m-1}.
\end{displaymath}
where $C$ is a constant which is independent of $f$ and $g$ and $[,]$ represents the usual commutator of the linear operators.
\end{lemma}

\noindent
Now, we verify the assumptions needed for Theorem \ref{Lwp}. We start with assumption (A1):
  \begin{lemma} \label{opA}
For any given $r>0$ it holds that for all $u \in  \mathrm B_r(0) \subseteq Y$, the linear operator $A(u)\colon X \to X$ with domain $D(A(u))\coloneqq \{ w \in H^{s-1}\colon A(u) w \in H^{s-1} \}$
generates a strongly continuous semigroup $T_u(t)$ in $X$ which satisfies
\be
\label{eq-SG}
    \| T_u(t) \|_{\mathcal L(X)} \leq \mathrm e^{\omega_r t} \quad \text{for all} \quad t\in [0,\infty)
\ee
for a uniform constant $\omega_r > 0$. i.e. The operator $A(u)$ given in (\ref{ql}), with domain
\begin{displaymath}
 \mathcal{D}(A)=\{\omega\in H^{s-1}:A(u)\omega\in H^{s-1}\}\subset H^{s-1}
\end{displaymath}
is quasi-m-accretive if $u\in H^{s}$, $s>\frac{5}{2}$.
\end{lemma}

\begin{proof}

A linear operator $A=A(u)$ in $X$ is \mbox{quasi-m-accretive} if \cite{KatoII}
 \begin{enumerate}
  \item [(a)] There is a real number $\beta$ such that $(A\omega,\omega)_X\geq -\beta||\omega||_X^2$ for all $\omega\in D(A)$.
  \item [(b)] The range of $A(u)+\lambda I$ is all of $X$ for some (or equivalently, all) $\lambda>\beta$.
 \end{enumerate}

Note that if the above property (a) holds, then $A+\lambda I$ is dissipative for all $\lambda>\beta$. Moreover, if $A$ is a closed operator, then $A+\lambda I$ has closed range in $H^{s-1}$ for all $\lambda>\beta$. Hence, in order to prove (b) in such a case, it is enough to show that $A+\lambda I$ has dense range in $H^{s-1}$ for all $\lambda>\beta$.  

First we show that for given $u\in H^s$, $A=a(u)\partial_x=a\partial_x$ is a well-defined linear operator on $H^{s-1}$.

Note that if $u\in H^s$, then $a=a(u)\in H^s$. The image $A w$ for a general $w\in H^{s-1}$ is uniquely determined since for $r>1/2$, the usual (pointwise) product $H^r\times H^r \to H^r$ is continuous, the product rule holds and  
\bew
(a w)_x = a_x w + a w_x  \quad \text{in} \; H^{s-2}. 
\eew
If $w\in D(A)$ we have that both $(a w)_x$ and $a_x w$ lie in $H^{s-1}$, hence 
\bew
A w = (a w)_x - a_x w = a w_x \in H^{s-1}.
\eew

Furthermore, $A$ is a closed operator.

Let $(v_n)_{n\in\N}$ be a sequence in $D(A)$ with $v_n\to v$ in $H^{s-1}$ and $A v_n \to w$ in $H^{s-1}$. Then $a v_n \in H^{s}$ for all $n\in \N$ by definition of $D(A)$ since an alternative way of writing the domain is $D(A)=\{\omega\in H^{s-1}:a(u)\omega\in H^{s}\}$ and $v_n\in D(A)$. Moreover, both $a v_n \to a v$ and $a_x v_n \to a_x v$ in $H^{s-1}$.
%by the continuity of the multiplication $H^r \times H^r \to H^r$ for $r>1/2$.
Therefore $(a v_n)_x \to w+a_x v$ in $H^{s-1}$. Having sequences $(a v_n)_{n\in\N}$ and $((a v_n)_x)_{n\in\N}$ convergent in $H^{s-1}$ implies that $(a v_n)_{n\in\N}$ converges in $H^s$ due to the following observations:\\
Both $(\Lambda^{s-1}(a v_n))_{n\in\N}$ and $(\Lambda^{s-1}(a v_n)_x)_{n\in\N}=(\partial_x\Lambda^{s-1}(a v_n))_{n\in\N}$ converge in $L^2$ , so $(\Lambda^{s-1}(a v_n))_{n\in\N}$ converges in $H^1$ and hence $(a v_n)_{n\in\N}$ converges in $H^s$.
Since the limit was already determined in $H^{s-1}$ it follows that $a v_n \to a v$ in $H^s$, thus $v\in D(A)$. Moreover the continuity of $\partial_x\colon H^s \to H^{s-1}$ implies that $\lim_{n\to\infty} (a v_n)_x = (a v)_x $, therefore $w=(a v)_x-a_x v = A v$.

Now, we take the following $H^{s-1}$ inner product
\begin{align}
 (A(u)\omega,\omega)_{s-1} &=(a(u)\partial_x\omega ,\omega)_{s-1} \nonumber\\
&=(\Lambda^{s-1}a(u)\partial_x\omega ,\Lambda^{s-1}\omega)_{0} \nonumber \\
&=([\Lambda^{s-1},a(u)]\partial_x\omega ,\Lambda^{s-1}\omega)_{0}+(a(u)\partial_x\Lambda^{s-1}\omega ,\Lambda^{s-1}\omega)_{0} \label{est}.
\end{align}
Using Cauchy-Schwartz's inequality and Lemma \ref{L-comm} with $m=s-1$, $\sigma=s$, we get the following estimate for the first term of (\ref{est}):
\begin{align*}
|([\Lambda^{s-1},a(u)]\partial_x\omega ,\Lambda^{s-1}\omega)_{0}|&\leq C||a(u)||_s||\partial_x \omega||_{s-2}||\omega||_{s-1}\\
&\leq \tilde{C}||\omega||_{s-1}^2
\end{align*}
for some constant $\tilde{C}$ depending on $||u||_s$.

For the second term of (\ref{est}), we refer to Lemma \ref{L-core} and use integration by parts to get:

\begin{align*}
|(a(u)\partial_x\Lambda^{s-1}\omega ,\Lambda^{s-1}\omega)_{0}|
&=|-\frac{1}{2}(a_x, (\Lambda^{s-1}\omega)^2)_0|\\
&\leq C ||a_x||_{L^\infty}||w||_{s-1}^2\\
&\leq \tilde{C}||\omega||_{s-1}^2
\end{align*}
Choosing $\beta=\tilde{c}||u||_{H^{s}}$, $\tilde{c}=\tilde{c}(s)$, the operator satisfies the inequality in (a). Thus, $A(u)+\lambda I$ is dissipative for all $\lambda>\beta$. Moreover, recall that $A(u)$ is a closed operator. Therefore, we now show that $A(u)+\lambda I$ has dense range in $H^{s-1}$ for all $\lambda>\beta$. 

It is known that the adjoint of an operator has trivial kernel, then the operator has dense range \cite{RS}. For $A(u)=a(u)\partial_x $, the adjoint operator can be expressed

\begin{displaymath}
A^*(u)=-\partial_x (a(u)).
\end{displaymath}

Observe that 
\bew
\partial_x a(u)\omega= (a(u)\omega)_x=a_x(u)\omega+a(u)\omega_x
\eew
and since $u_x\in L^{\infty}$ and $\omega\in H^{s-1}$, we have $a_x(u)\omega\in H^{s-1}$. \\
Having also $a(u)\omega_x\in H^{s-1}$ reveals that
\begin{align*}
&\mathcal{D}(A)=\{\omega\in H^{s-1}:A(u)\omega\in H^{s-1}\}=\\
&\mathcal{D}(A^*)=\{\omega\in H^{s-1}:A^{*}(u)\omega\in H^{s-1}\}.
\end{align*}

Now, assume to the contrary that $A(u)+\lambda I$ does not have a dense range in $H^{s-1}$. Then, there exists $0\neq z\in H^{s-1}$
such that $((A(u)+\lambda I)\omega,z)_{s-1}=0$ for all $\omega\in \mathcal{D}(A)$. Since $H^s\subset \mathcal{D}(A)$, $\mathcal{D}(A)=\mathcal{D}(A^*)$ is dense in $H^{s-1}$. It means that there exists a sequence $z_k\in \mathcal{D}(A^*)$ such that it converges to an element $z\in H^{s-1}$. Recall that $A^*$ is closed. So, $z\in \mathcal{D}(A^*)$. Moreover,
\begin{equation*}
((A(u)+\lambda I)\omega,z)_{s-1}=(\omega,(A(u)+\lambda I)^*z)_{s-1}=0
\end{equation*}
reveals that $A^*(u)+\lambda z=0$ in $H^{s-1}$. Multiplying by $z$ and integrating by parts, we get
\begin{equation*}
0=((A^*(u)+\lambda I)z,z)_{s-1}=(\lambda z,z)_{s-1}+(z,A(u)z)_{s-1}\geq (\lambda-\beta)||z||_{s-1}^2~~~\forall \lambda>\beta
\end{equation*}
and thus, $z=0$ which contradicts our assumption. It completes the proof of (b). Therefore, the operator $A(u)$ is quasi-m-accretive.
\end{proof}

In the proof of Lemma \ref{opA} we used the following fact that $\mathcal C^{\infty}$ is a \emph{core} for $A$ in $H^{s-1}$, i.e. $A(u)v$ can be approximated by smooth functions in $H^{s-1}$  (similar arguments done for CH can be found in \cite{CE2}): 

\begin{lemma}
\label{L-core}
    Given $v\in D(A)$ there exists a sequence $(v_n)_n$ in $\mathcal C^{\infty}$ such that both $v_n \to v$ and $A v_n \to Av$ in $H^{s-1}$.
\end{lemma}

\begin{proof}
Let $v\in D(A)$ and fix $\rho\in C_c^\infty$, with $\rho \geq 0$ and $\int_{\R} \rho =1$. Given $n\geq 1$, let $\rho_n= n\rho(nx)$. If we set $v_n\coloneqq \rho_n * v$, then $v_n \in C_c^\infty$ for $n\geq 1$ and $v_n \to v$ in $H_p^{s-1}$.
We have to prove that $(av_n)_x \to (av)_x$ in $H^{s-1}$. Since $v\in D(A)$, we have that $av_x\in H^{s-1}$ and hence 
 $\rho_n *(av_x)\rightarrow av_x$. Moreover, since $av_n\in H^s$ it follows that $(av_n)_x = a_x v_n + a(v_n)_x \in H^{s-1}$,  hence we have that  $a_xv_n\rightarrow a_x v$ in $H^{s-1}$. Therefore
\begin{align*}
    (av_n)_x-(av)_x&=a_xv_n-a_xv +\rho_n *(av_x)-av_x + a(v_n)_x - \rho_n*(av_x)
\end{align*}
holds true and it suffices to show that $ a(v_n)_x-\rho_n*(av_x) \to 0$ in $H^{s-1}$. To this end, denote 
$$P_n v \coloneqq a(v_n)_x-\rho_n*(av_x), \quad  n\geq 1.$$ 
We will show   that there exists $K>0$ independent of $v$ such that 
\be
\label{eq-Pn} 
\|P_n v \|_{s-1} \leq K\|v\|_{s-1}, \quad n\geq 1. 
\ee
That will enable us to conclude that $P_n$ is uniformly bounded in $H^{s-1}$ by uniform boundedness principle. When we approximate $v$ in $H^{s-1}$ by smooth functions, and use this conclusion, we will be able to prove the assertion $P_n\to 0$ for $v\in C_c^\infty$. Since the set of smooth functions is dense in $H^{s-1}$ and $P_n$ are uniformly bounded, the proof will be completed. 

We first notice that 
\begin{align*}
P_n v(x) &= \int_{\R}(\rho_n)_y(y) (a(x) - a(x-y))v(x-y) dy +(\rho_n*(a_x v))(x)\\
             &= n^2\int_{\R} \rho_y(ny) (a(x)-a(x-y)) v(x-y)dy+(\rho_n*(a_x v))(x)\\
             &= n\int_{-1}^1 \rho_y(y) (a(x)-a(x-\frac{y}{n})) v(x-\frac{y}{n})dy+(\rho_n*(a_x v))(x),
\end{align*}
where $\rm supp(\rho)\subset [-1,1]$. Moreover, using the mean value theorem, we obtain the estimate 
\begin{align*}
|n^2\int_{\R} \rho_y(ny) (a(x)-a(x-y)) v(x-y)dy|
        &= |n^2\int_{\R} \rho_y(ny) a_x(x_0) y v(x-y)dy|\\
        &= |\int_{-1}^1\rho_y(y) a_x(x_0) y v(x-y)dy|\\
        &\leq ||a_x||_{L^\infty}\int_{-1}^1 |\rho_y(y)|\,|y|\,|v(x-\frac{y}{n})|dy.
\end{align*}
for some $x_0\in (x,x-y)$. Let now  $C\coloneqq \sup_{x\in\mathbb{R}}||a_x||_{L^\infty}^2\int_{-1}^1 |\rho_y(y)y|^2dy$. Then the Cauchy-Schwarz inequality, Fubini's theorem and the fact that the operator $\Lambda^{s-1}$ commutes with integration yield that
\begin{align*}
&||n^2\int_{\R} \rho_y(ny) (a(x)-a(x-y)) v(x-y)dy||^2_{s-1}\\
&=||\Lambda^{s-1}\int_{-1}^1 \rho_y(y) a_x(x_0)y (v(x-\frac{y}{n}))dy||^2_{2}\\
%&=\left(\int \left|\Lambda^{s-1}\int_{-1}^1\rho_y(y) a_x(x_0)y (v(x-\frac{y}{n}))dy\right|^2 dx\right)^{1/2}\\
&=\int_{\R} \left|\int_{-1}^1 \rho_y(y) a_x(x_0)y \Lambda^{s-1}v(x-\frac{y}{n})dy\right|^2 dx\\
&\leq C \int_{-1}^1\int |\Lambda^{s-1}v(x-\frac{y}{n})|^2 dx dy\\
&\leq 2C||v||_{s-1}
\end{align*}
Moreover, we obtain by Plancherel's theorem that 
\begin{align*}
||\rho_n*(a_x v)||_{s-1} &= ||\Lambda^{s-1}(\rho_n*(a_x v))||_{2} 
        = ||\rho_n*\Lambda^{s-1}(a_x v))||_{2} 
%    &=||(1+\xi^2)^{\frac{s-1}{2}}\widehat{\rho_n}\widehat{a_x v}||_{2}\\
%    &=||\widehat{\rho_n}(1+\xi^2)^{\frac{s-1}{2}}\widehat{a_x v}||_{2}\\
  %||\widehat{\rho_n}||_{2}||(1+\xi^2)^{\frac{s-1}{2}}\widehat{a_x v}||_{2} 
    &\leq     ||\Lambda^{s-1}(a_x v)||_2\\
%    &= ||\rho_n||_{H^{s-1}}||a_x v||_{H^{s-1}}\\
    &\leq  ||a_x||_{L^\infty}||v||_{s-1}
\end{align*}
Therefore we may conclude that
\begin{equation}\label{bound}
||P_n v||_{s-1}\leq (\sqrt{2C}+||a_x||_{L^\infty})\,||v||_{s-1},~~n\geq 1.
\end{equation}
For $K=\sqrt{2C}+||a_x||_{L^\infty}$ in (\ref{eq-Pn}), proof is completed by the estimate (\ref{bound}).
\end{proof}

We continue with the proof of assumption (A2):

\begin{lemma}
$A$ maps $Y$ into $\mathcal L(Y,X)$, more precisely the domain $D(A(u))$ contains $Y$ and the restriction $A(u)|_Y$ belongs to $\mathcal L(Y,X)$ for any $u\in Y$. Furthermore $A$ is Lipschitz continuous in the sense that for all $r>0$ there exists a constant $C_1$ which only depends on $r$ such that
\begin{equation*}
\| A(u) - A(v) \|_{\mathcal L(Y,X)} \leq C_1 \, \|u-v\|_X 
\end{equation*}
for all $u,~v$ inside $\mathrm B_r(0) \subseteq Y$.
\end{lemma}

\begin{proof}
Given $u,~v,~z\in H^{s}$ with $s>\frac{5}{2}$,

\begin{align*}
||(A(u)-A(v))z||_{s-1}&=||((u-v)+\Lambda^{-2\nu} [(u-v),(-\partial_x^2)^{\nu}])\partial_x z||_{s-1}\\
& \leq C( ||(u-v)\partial_x z||_{s-1}+||\Lambda^{-2\nu} [(u-v),(-\partial_x^2)^{\nu}]\partial_x z||_{s-1})\\
&= C( ||(u-v)\partial_x z||_{s-1}+||[(u-v),(-\partial_x^2)^{\nu}]\partial_x z||_{s-1-2\nu})\\
&\leq C\big(||u-v||_{s-1}||\partial_x z||_{L^\infty}+||(u-v)||_{s-1}||\partial_x z||_{s-2}\big)\\
&\leq C ||u-v||_{s-1}|| z||_{s}
\end{align*}
in view of Lemma \ref{L-comm} for $m=2\nu$, $\sigma=s-1$. 

% Observe that 
%\begin{align*}
%||[\Lambda^{-2\nu},u-v](-\partial_x^2)^{\nu}\partial_x \omega||_{s-1}&=||\Lambda^{s-1}([\Lambda^{-2\nu},u-v](-\partial_x^2)^{\nu}\partial_x \omega)||_{L^2}\\
%&= ||\Lambda^{s-1-2\nu}(u-v)(-\partial_x^2)^{\nu}\partial_x \omega-\Lambda^{s-1}(u-v)\Lambda^{-2\nu}(-\partial_x^2)^{\nu}\partial_x \omega||_{L^2}\\
%&=||[\Lambda^{s-1-2\nu},(u-v)](-\partial_x^2)^{\nu}\partial_x \omega+(u-v)\Lambda^{s-1-2\nu}(-\partial_x^2)^{\nu}\partial_x \omega\\&-\Lambda^{s-1}(u-v)\Lambda^{-2\nu}(-\partial_x^2)^{\nu}\partial_x \omega||_{L^2}\\
%&=||[\Lambda^{s-1-2\nu},(u-v)](-\partial_x^2)^{\nu}\partial_x \omega-[\Lambda^{s-1},(u-v)]\Lambda^{-2\nu}(-\partial_x^2)^{\nu}\partial_x \omega||_{L^2}
%\end{align*}
%
\end{proof}

Now, we define a bounded linear operator and prove assumption (A3):
\begin{lemma}
For any $u\in Y$ there exists a bounded linear operator $B(u) \in \mathcal L(X)$ satisfying $B(u) = \Lambda A(u) \Lambda^{-1} - A(u)$ and $B \colon Y \to \mathcal L(X)$ is uniformly bounded on bounded sets in $Y$. Furthermore for all $r>0$ there exists a constant $C_2$ which depends only on $r$ such that  
\bew
\| B(u) - B(v)\|_{\mathcal L(X)} \leq C_2 \, \|u-v\|_Y
\eew
for all $u,v \in \mathrm B_r(0)\subseteq Y$. Here, $A(u)$ is the operator given by (\ref{ql}).
\end{lemma}

\begin{proof}
Note that
\begin{align*}
\Lambda A(u) \Lambda^{-1} - A(u)&=\Lambda a(u)\partial_x \Lambda^{-1} - a(u)\partial_x=[\Lambda,a(u)]\Lambda^{-1}\partial_x
\end{align*}
since $\partial_x$ and $\Lambda$ commute. Therefore, for $w\in H^{s-1}$,

\begin{align*}
||B(u)w||_{s-1}&=||[\Lambda,a(u)]\Lambda^{-1}\partial_x\omega||_{s-1}\\
&\leq C ||a(u)||_s||\Lambda^{-1}\partial_x\omega||_{s-1}\\
&\leq C ||a(u)||_s||\omega||_{s-1}\\
&\leq \tilde{C} ||\omega||_{s-1}
\end{align*}
for some constant $\tilde{C}$ which depends on $||u||_s$. Here, we use Lemma \ref{L-comm} with $m=1$, $\sigma=s$. Moreover,
\begin{align*}
||(B(u)-B(v))w||_{s-1}&=||[\Lambda,a(u)-a(v)]\Lambda^{-1}\partial_x\omega||_{s-1}\\
&\leq C ||a(u)-a(v)||_s||\Lambda^{-1}\partial_x\omega||_{s-1}\\
&\leq C ||a(u)-a(v)||_s||\omega||_{s-1}\\
&\leq C (||u-v||_s+||\Lambda^{-2\nu} [(u-v),(-\partial_x^2)^{\nu}]||_s)||\omega||_{s-1}\\
&\leq C (||u-v||_s+||[(u-v),(-\partial_x^2)^{\nu}]||_{s-2\nu})||\omega||_{s-1}\\
&\leq C_2 ||\omega||_{s-1}
\end{align*}
where $C_2$ is a constant depending on $||u||_s$, $||v||_s$. 
\end{proof}

The last assumption (A4) is proved below:

\begin{lemma}
The map $f\colon Y \to Y$ is locally $X$-Lipschitz continuous in the sense that for every $r>0$ there exists a constant $C_3>0$, depending only on $r$, such that
\bew
\| f(u) - f(v)\|_{X} \leq C_3 \, \|u - v\|_X \quad \text{for all} \; u,v \in \mathrm B_r(0) \subseteq Y
\eew 
and locally $Y$-Lipschitz continuous in the sense that for every $r>0$ there exists a constant $C_4>0$, depending only on $r$, such that
\bew
\| f(u) - f(v)\|_{Y} \leq C_4 \, \|u - v\|_Y \quad \text{for all} \; u,v \in \mathrm B_r(0) \subseteq Y.
\eew 

\end{lemma}

\begin{proof}
Observe that
\begin{displaymath}
f(u)= \Lambda^{-2\nu}\partial_x (u^2).
\end{displaymath}
Therefore,
\begin{align*}
||f(u)-f(v)||_{s-1} &\leq C ||u^2-v^2||_{s-2\nu}\\
                    &=    C ||(u-v)(u+v)||_{s-2\nu}\\
                    &\leq C ||(u+v)||_{L^\infty}||(u-v)||_{s-2\nu}\\
            		&\leq C ||(u+v)||_{s}||(u-v)||_{s-2\nu}\\
            		&\leq C_3 ||u-v||_{s-1}
\end{align*}
where $C_3$ is a constant depending on $||u||_{H^{s}}$ and $||v||_{H^{s}}$. This proves (ii).

Similar arguments will show that we have the following estimates:
\begin{align*}
||f(u)-f(v)||_{H^{s}}\leq& C_4 ||u-v||_{H^{s}}
\end{align*}
where $C_4$ is also a constant depending on $||u||_{H^{s}}$ and $||v||_{H^{s}}$. Since we choose $u_0\in H^{s}$, this estimate actually corresponds to the proof of continuous dependence on the initial data. Note that (i) can be obtained from (iii) by choosing $v=0$. Hence, we get the estimates for (A4).
\end{proof}

%\begin{remark}
%\label{regularity}
%As a consequence of equation (\ref{fCH}), the solutions $ u_t\in C([0,T), H^{s-1})$, $s>3/2$, are more regular since $u_x\in H^{s-1}$ and there exists no other terms involving higher order derivatives with respect to the spatial variable.
%\end{remark}

%\setcounter{equation}{0}
%\section{Wave Breaking}
%\noindent
%
%\begin{proposition}\label{brk}
%    If, for some initial data $u_0\in H^2$, the maximal existence time $T>0$ of the solution to (\ref{FCH}) is finite, then the solution $u(x,t)\in C([0,T), H^2)\cap C^1([0,T),L^{2})$ has the property that
%    \begin{equation}  \label{sup}
%        \sup_{t\in[0,T),~x\in\mathbb{R}}\{|u(x,t)|\}<\infty
%    \end{equation}
%    whereas
%    \begin{equation}\label{slope}
%     \limsup_{t\uparrow T}\{(-\partial_x^2)^{\nu/2}u(x,t)\}=\infty
%    \end{equation}
%\end{proposition}
%\begin{proof}
%Multiplying equation (\ref{FCH}) by $u$ and integrating over $\mathbb{R}$, we find that
%\begin{equation}
%  \label{E}
%    E(u)= \frac{1}{2}\int_{\mathbb{R}} \Big(u^2 +\frac{5}{4} ((-\partial_x^2)^{\nu/2}u)^2\Big) dx
%  \end{equation}
%is a conserved quantity of (\ref{FCH}).
%\end{proof}
%
%
%We need to check the wave breaking condition so that we may continue the proof...
%\vspace*{30pt}

%\section*{References}

\end{document}